\documentclass[12pt, reqno]{amsart}
\usepackage{amsmath, amsthm, amscd, amsfonts, amssymb, graphicx, color}
\usepackage[bookmarksnumbered, colorlinks, plainpages]{hyperref}

\newtheorem{theorem}{Theorem}[section]
\newtheorem{lemma}[theorem]{Lemma}
\newtheorem{proposition}[theorem]{Proposition}

\theoremstyle{definition}
\newtheorem{definition}[theorem]{Definition}

\theoremstyle{remark}
\newtheorem{remark}[theorem]{Remark}
\numberwithin{equation}{section}
\begin{document}
\begin{center}
\large{\textbf{On Riemannian Poisson warped product space}}
\end{center}
\begin{center}
	{Buddhadev Pal* and Pankaj Kumar\footnote{Second author's work is funded by UGC, India in the form of JRF [1269/(SC)(CSIR-UGC NET DEC. 2016)].} }
\end{center}
\vskip 0.3cm
\vskip 0.3cm
\begin{center}
	Department of Mathematics, Institute of Science, Banaras Hindu University, Varanasi-221005, India.\\
	*Email Id: pal.buddha@gmail.com\\
	Email Id: pankaj.kumar14@bhu.ac.in
\end{center}
\vskip 0.5cm
\begin{center}
	\textbf{Abstract}
\end{center}
A formal treatment of Killing 1-form and 2-Killing 1-form on Riemannian Poisson manifold, Riemannian Poisson warped product space are presented. In this way, we obtain Bochner type results on compact Riemannian Poisson manifold, compact Riemannian Poisson warped product space for Killing 1-form and 2-Killing 1-form. Finally, we give the characterization of a 2-Killing 1-form on $(\mathbb{R}^2,g,\Pi)$.
\section{Introduction}
To provide the example of Riemannian spaces having negative curvature Bishop and O'Neill \cite{lota} introduced the notion of warped space. From then on original and generalized forms of warped product spaces have been widely discussed by both mathematicians and physicists \cite{kim,kim1,kim2,rim,ab,qu,dob,ppb}.\par Let $(\tilde{M_1},\tilde{g_{1}})$ and $(\tilde{M_2},\tilde{g_{2}})$ are two pseudo-Riemannian manifolds with positive smooth function $f$ on $\tilde{M_1}$. Let $\pi_1 : \tilde{M_1} \times \tilde{M_2} \rightarrow \tilde{M_1}$ and $\pi_2 : \tilde{M_1} \times \tilde{M_2} \rightarrow \tilde{M_1}$ are the projections. The warped product $\tilde{M}=\tilde{M_1} \times_f \tilde{M_2}$ is the product manifold $\tilde{M_1} \times \tilde{M_2}$ endowed with the metric tensor $$\tilde{g}^f=\pi_1^*(\tilde{g_{1}})+ (f\circ\pi_1)^2\pi_2^*(\tilde{g_{2}}),$$
called warped product and the ordered-pair $(\tilde{M},\tilde{g}^f)$ known as warped product space. Here $\tilde{M_1}$, $\tilde{M_2}$ and $f$ are respectively known as base space, fiber space and warping function of the warped product space $(\tilde{M},\tilde{g}^f)$ and $^*$ stand for pull-back operator.\par
Killing vector fields are the relevant object for the geometry specially in pseudo-Riemannian geometry where mathematicians characterized the existence of Killing vector fields. Killing vector fields are also studied by many physicists in the prospective of general relativity in which these are expounded in the term of symmetry. Bochner \cite{sbo,sbo1,sbo2}, studied in detail Killing vector fields and provided various remarkable results. K. Yano \cite{ko1,ko2}, consider a compact orientable Riemannian spaces with boundary and generalized the Bochner technique to study Killing vector fields on it. S. Yorozu \cite{sz1,sz2}, discussed the non-existence of Killing vector fields on complete Riemannian spaces and also did the same for non-compact Riemannian spaces with boundary. Generalized form of Killing vector fields like conformal vector fields, 2-Killing vector fields have been investigated in \cite{sd1,sd2,sd3,kdr,wkh,shn}, for ambient spaces. T. Opera \cite{top}, introduced the perception of 2-Killing vector fields and provided the relation between curvature, monotone vector fields and 2-Killing vector fields on Riemannian spaces. Moreover, characterized the 2-Killing vector field on $\mathbb{R}^n$. S. Shenawy and B. $\ddot{U}nal$ \cite{ssb}, provided some results of 2-Killing vector field for warped product space and apply these results to characterize it on some famous warped space time model.
Z. Erjavec \cite{zec}, currently characterized proper conformal Killing vector fields and determine some proper 2-Killing vector fields in Sol space.\par
Poisson \cite{sdp}, introduced a bracket as a tool for classical dynamics and Lie \cite{lie}, explored the geometry of this bracket. In \cite{ivs,jpn}, authors adopted the Poisson structure and provided the notion of Poisson manifold. The geometric notions like connection, curvatures, metric etc., were discussed in \cite{rfl,mbo,mbo2,zs}, on Poisson manifold. In \cite{rnm}, authors formulated several concepts on product manifold like product Poisson tensor and product Riemannian metric . In \cite{yar}, authors discussed the some geometric notions like contravariant Levi-Civita connection, Riemann and Ricci curvatures on the product of two pseudo-Riemannian spaces which is associated with the product Poisson structure, warped bivector field.\par
The aim of this article is to provide the notions of Killing 1-form and 2-Killing form and try to study these two notions on Riemannian Poisson manifold and Riemannian Poisson warped product space.\par
The outline of this article is as follows. In Section 2, we look back on some classical notions like cometric, curvatures, contravariant Levi-Civita connection on Poisson manifold and give the definition of Riemannian Poisson manifold $(\tilde{M},g,\Pi)$. Moreover, we provide the explicit form of cometric $g^f$ and contravariant Levi-Civita connection $\mathcal{D}$ on $(M=\tilde{M_1}\times_f\tilde{M_2}, g^f)$. In Section 3, we characterize the Killing 1-form on Riemannian Poisson manifold and Riemannian Poisson warped product space $(M=\tilde{M_1}\times_f \tilde{M_2}, g^f,\Pi)$. Moreover, we introduce the concept of parallel 1-form and provide Bochner type results on compact Riemannian Poisson manifold and Riemannian Poisson warped product space in Theorems 3.10, 3.11. In Section 4, we study 2-Killing 1-first form and characterize 2-Killing 1-form on $\mathbb{R}^2$ in Theorem 4.8.
\section{Preliminaries}
\subsection{Geometric structure on Poisson manifold}
Lots of basic terms and consequence related to Poisson manifold presented in \cite{ivs}. Let $\tilde{M}$ be a manifold. A Lie bracket map $\{.,.\} : \mathcal{C}^\infty(\tilde{M})\times\mathcal{C}^\infty(\tilde{M})\rightarrow\mathcal{C}^\infty(\tilde{M})$
 is said to be Poisson bracket on $\tilde{M}$ if it follows the Leibniz identity i.e.,
\begin{equation*}
\{\phi_1,\phi_2\phi_3\}=\{\phi_1,\phi_2\}\phi_3+\phi_2\{\phi_1,\phi_3\},\qquad\forall\:\:\phi_1,\phi_2,\phi_3\in\mathcal{C}^\infty(\tilde{M}).
\end{equation*}
The pair $(\tilde{M},\{.,.\})$ is said to be Poisson manifold.\\
Let $(\tilde{M},\{.,.\})$ be Poisson manifold and $\phi_1\in\mathcal{C}^\infty(\tilde{M})$ then we can find a unique vector field $X_{\phi_1}$ on $\tilde{M}$ associate to $\phi_1$ such that 
\begin{equation*}
X_{\phi_1}(\phi_2)=\{\phi_2,\phi_1\},\qquad\forall\:\:\phi_2\in\mathcal{C}^\infty(\tilde{M}).
\end{equation*}
The vector field $X_{\phi_1}$ is said to be Hamiltonian vector field of the function $\phi_1$. If
\begin{equation*}
X_{\phi_1}(\phi_2)=0,\qquad\forall\:\:\phi_2\in\mathcal{C}^\infty(\tilde{M}),
\end{equation*}
then $\phi_1\in\mathcal{C}^\infty(\tilde{M})$ is called Casimir function. The Leibniz identity also guarantee the existence of a bivector field $\Pi\in\mathfrak{X}^2(\tilde{M})=\Gamma(\Lambda^2T\tilde{M})$ such that
\begin{equation*}
\{\phi_1,\phi_2\}=\Pi(d\phi_1,d\phi_2),\qquad\forall\:\:\phi_1,\phi_2\in\mathcal{C}^\infty(\tilde{M}).
\end{equation*}
A bracket $[.,.]_S$ on $(\tilde{M},\{.,.\})$ considered to be Schouten bracket associated with bivector field $\Pi$ if
\begin{equation*}
\frac{1}{2}[\Pi,\Pi]_S(d\phi_1,d\phi_2,d\phi_3)=\{\{\phi_1,\phi_2\},\phi_3\}+\{\{\phi_2,\phi_3\},\phi_1\}+\{\{\phi_3,\phi_1\},\psi_2\}.
\end{equation*}
A bivector field $\Pi$ on $\tilde{M}$ is called Poisson tensor if $[\Pi,\Pi]_S=0$.\\\\
\textbf{Remark:} Many authors assume $(\tilde{M},\Pi)$ as a Poisson manifold when $\Pi$ is a Poisson tensor. Here, we consider same notion.\\\\
Let us assume that if $\tilde{M}$ is a manifold with a bivector field $\Pi$ then there is a natural homomorphism $\sharp_{\Pi}:T^*\tilde{M}\rightarrow TM$ corresponding to $\Pi$ given by
\begin{equation*}
\eta(\sharp_{\Pi}(\omega))=\Pi(\omega,\eta),\qquad\forall\:\:\omega,\eta\in T^*\tilde{M},
\end{equation*}
called  sharp map(anchor map).\par
If $\Pi$ is a bivector field on $\tilde{M}$, then it give rise to a bracket $[.,.]_\Pi$ on smooth 1-forms $\Gamma(T^*\tilde{M})$ is said to be Koszual bracket defined by
\begin{equation*}
[\omega,\eta]_\Pi=\mathcal{L}_{\sharp_{\Pi}(\omega)}\eta-\mathcal{L}_{\sharp_{\Pi}(\eta)}\omega-d(\Pi(\omega,\eta)).
\end{equation*}\par
Let $(\tilde{M},\Pi)$ be a Poisson manifold where $\Pi$ is a Poisson tensor on $\tilde{M}$ then Koszual bracket $[.,.]_\Pi$ convert into usual Lie bracket.\par
If $\sharp_{\Pi}$ is the sharp map on Poisson manifold $(\tilde{M},\Pi)$, then there is a Lie algebra homomorphism $\sharp_{\Pi}:\Gamma(T^*\tilde{M})\rightarrow \Gamma(T\tilde{M})$, such that
\begin{equation*}
\sharp_{\Pi}([\omega,\eta]_\Pi)=[\sharp_{\Pi}(\omega),\sharp_{\Pi}(\eta)],
\end{equation*}
where $[.,.]$ is the usual Lie bracket on $\Gamma(T\tilde{M})$.\\\\
Let $(\tilde{M},\Pi)$ be a Poisson manifold. In \cite{rfl}, authors introduced the concept of contravariant connection $\mathcal{D}$ on $\tilde{M}$. The torsion and curvature tensors corresponding to this connection $\mathcal{D}$ are given by
\begin{align*}
\mathcal{T}(\omega,\eta)&=\mathcal{D}_\omega\eta-\mathcal{D}_\eta\omega-[\omega,\eta]_\Pi,\\
\mathcal{R}(\omega,\eta)\gamma&=\mathcal{D}_\omega\mathcal{D}_\eta\gamma-\mathcal{D}_\eta\mathcal{D}_\omega\gamma-\mathcal{D}_{[\omega,\eta]_\Pi}\gamma,
\end{align*}
where $\mathcal{T}$ is $(2,1)$-type tensor and $\mathcal{R}$ is $(3,1)$-type tensor. Here $\mathcal{D}$ is said to be torsion-free if $\mathcal{T}=0$ and flat if $\mathcal{R}=0.$\\\\
Let $(M,\tilde{g})$ be a pseudo-Riemannian manifold. The bundle isomorphism\\ $\flat_{\tilde{g}}:TM\rightarrow T^*M$ is a map such that $X\mapsto \tilde{g}(X,.)$ and its inverse map
\begin{align*}
\sharp_{\tilde{g}}:&T^*M\rightarrow TM\\
&\omega\mapsto \sharp_{\tilde{g}}(\omega)
\end{align*} 
such that $\omega(X)=\tilde{g}(\sharp_{\tilde{g}}(\omega),X)$. The metric $g$ on the cotangent bundle $T^*M$ is defined by
\begin{equation*}
g(\omega,\eta)=\tilde{g}(\sharp_{\tilde{g}}(\omega),\sharp_{\tilde{g}}(\eta)).
\end{equation*}
This metric $g$ is said to be cometric of the metric $\tilde{g}$.\par
Let $(\tilde{M},\Pi)$ be a Poisson manifold and $g$ is cometric then there exists a unique contravariant connection $\mathcal{D}$ on $\tilde{M}$ characterized by
\begin{align}
2g(\mathcal{D}_\omega\eta,\gamma)&=\sharp_{\Pi}(\omega)g(\eta,\gamma)+\sharp_{\Pi}(\eta)g(\omega,\gamma)-\sharp_{\Pi}(\gamma)g(\omega,\eta)\nonumber\\
&+g([\omega,\eta]_\Pi,\gamma)-g([\eta,\gamma]_\Pi,\omega)+g([\gamma,\omega]_\Pi,\eta),
\end{align}
for any $\omega,\eta,\gamma\in\Omega^1(\tilde{M})$,
and follows the following two conditions
\begin{align*}
(\textbf{i}).&\:\mathcal{D}_\omega\eta-\mathcal{D}_\eta\omega=[\omega,\eta]_\Pi\text{(Torsion-free)},\\
(\textbf{ii}).&\:\sharp_{\Pi}(\omega)g(\eta,\gamma)=g(\mathcal{D}_\omega\eta,\gamma)+g(\eta,\mathcal{D}_\omega\gamma)\text{(Metric condition)},
\end{align*}
for any $\omega,\eta,\gamma\in\Omega^1(\tilde{M})$. Contravariant connection $\mathcal{D}$ with properties $(i)$ and $(ii)$ is said to be contravariant Levi-Civita connection associated to pair $(\Pi,g)$ on $\tilde{M}$.\par 
Let $(\tilde{M},\Pi)$ be a $n$-dimensional Poisson manifold with connection $\mathcal{D}$ and $p$ is any point on $\tilde{M}$. The Ricci curvature $Ric_p$ and scalar curvature at $p$ corresponding to the local orthonormal coframe $\{\theta_1,...,\theta_n\}$ of $T_p^*\tilde{M}$, given by
\begin{align}
Ric_p(\omega,\eta)&=\displaystyle\sum_{i=1}^{n}g_p(\mathcal{R}_p(\omega,\theta_i)\theta_i,\eta),\\
S_p&=\displaystyle\sum_{i=1}^{n}Ric_p(\theta_i,\theta_i),
\end{align}
for any $\omega,\eta\in T_p^*\tilde{M}$.\\\\
Let $(\tilde{M},\Pi)$ be a Poisson manifold with connection $\mathcal{D}$ and $f\in\mathcal{C}^\infty(\tilde{M})$ then $\mathcal{D}f=df\circ\sharp_\Pi\in\mathfrak{X}^1(\tilde{M})$, defined by
\begin{equation*}
(\mathcal{D}f)(\omega)=\mathcal{D}_\omega f=\sharp_\Pi(\omega)(f)=df(\sharp_\Pi(\omega)),
\end{equation*}
for any $\omega\in\Omega^1(\tilde{M}).$\\\\
Let $(\tilde{M},\Pi)$ be a Poisson manifold with connection $\mathcal{D}$ satisfies $\mathcal{D}\Pi=0$ i.e.,
\begin{equation*}
\sharp_{\Pi}(\omega)\Pi(\eta,\gamma)-\Pi(\mathcal{D}_\omega\eta,\gamma)-\Pi(\eta,\mathcal{D}_\omega\gamma)=0,
\end{equation*}
for any $\omega,\eta,\gamma\in\Omega^1(M)$, then triplet $(\tilde{M},g,\Pi)$ called Riemannian Poisson manifold.\par
Let $(\tilde{M},\Pi)$ be a Poisson manifold with cometric $g$, then field endomorphism $$J:T^*\tilde{M}\rightarrow T^*\tilde{M}$$ provides 
\begin{equation*}
\Pi(\omega,\eta)=g(J\omega,\eta)=-g(\omega,J\eta),
\end{equation*}
for any $\omega,\eta\in T^*\tilde{M}.$\par
Let $(\tilde{M},g,\Pi)$ be a Riemannian Poisson manifold and $J$ is a field endomorphism on $\tilde{M}$ then $\mathcal{D}J=0$ i.e,
\begin{equation*}
\mathcal{D}_\omega(J\eta)=J\mathcal{D}_\omega\eta,
\end{equation*}
for any $\omega,\eta\in T^*\tilde{M}.$
\subsection{Cometric and contravariant Levi-Civita connection on warped product space}
The explicit form of the warped metric
\begin{equation*}
\tilde{g}^f=\pi^*(\tilde{g}_1)+(f^h)^2\sigma^*(\tilde{g}_2),
\end{equation*} on $(\tilde{M_1},\tilde{g}_1)$ and $(\tilde{M_2},\tilde{g}_2)$ is given by
\begin{eqnarray}
\left\{
\begin{array}{ll}
\tilde{g}^f(X_1^h,Y_1^h)=\tilde{g}_1(X_1,Y_1)^h,\\
\tilde{g}^f(X_1^h,Y_2^v)=\tilde{g}_1(X_2^v,Y_1^h)=0,\\
\tilde{g}^f(X_2^v,Y_2^v)=(f^h)^2\tilde{g}_2(X_2,Y_2)^v,
\end{array}
\right.
\end{eqnarray}
for any $X_1,Y_1\in\Gamma(T\tilde{M_1})$ and $X_2,Y_2\in\Gamma(T\tilde{M_2})$. Here $f\circ\pi=f^h$ is horizontal lift of $f$ from $\tilde{M_2}$ to $\tilde{M_1}\times \tilde{M_2}$. For more detail of horizontal and vertical lifts on product space see in (\cite{rnm,ddo,lot}).\\
As a consequence of the Proposition 3.3 of (\cite{ddo},p. 23), one has the following proposition which provides explicit form to the cometric
\begin{equation*}
g^f=g_1^h+\frac{1}{(f^h)^2}g_2^v,
\end{equation*}
 of warped metric $\tilde{g}^f$.
\begin{proposition}
Let two pseudo-Riemannian manifolds be  $(\tilde{M_1},\tilde{g}_1)$ and $(\tilde{M_2},\tilde{g}_2)$ and a smooth function be $f:\tilde{M_1}\rightarrow\mathbb{R}^+$. Then cometric $g^f$ of the metric $\tilde{g}^f$ is explicitly can be written as
\begin{eqnarray}
\left\{
\begin{array}{ll}
g^f(\omega_1^h,\eta_1^h)=g_1(\omega_1,\eta_1)^h,\\
g^f(\omega_1^h,\eta_2^v)=g_1(\omega_2^v,\eta_1^h)=0,\\
g^f(\omega_2^v,\eta_2^v)=\frac{1}{(f^h)^2}g_2(\omega_2,\eta_2)^v,
\end{array}
\right.
\end{eqnarray} 
for any $\omega_1,\eta_1\in\Gamma(T^*\tilde{M_1})$ and $\omega_2,\eta_2\in\Gamma(T^*\tilde{M_2})$. Where $g_1$ and $g_2$ are the cometric of the metric $\tilde{g}_1$ and $\tilde{g}_2$ respectively.
\end{proposition}
The ordered pair $(\tilde{M}=\tilde{M_1}\times_f \tilde{M_2}, g^f)$ is said to be contravariant warped product space of warped space $(\tilde{M}=\tilde{M_1}\times_f \tilde{M_2}, \tilde{g}^f)$.\\\par
Contravariant Levi-Civita connection $\mathcal{D}$ associated with pair $(g^f,\Pi)$ $(\text{where}\  g^f=g_1^h+\frac{1}{(f^h)^2}g_2^v\ \text{and}\ \Pi=\Pi_1+\Pi_2)$ on $(\tilde{M}=\tilde{M_1}\times_f \tilde{M_2}, g^f)$ is given by proposition:
\begin{proposition}
	For any $\omega_1,\eta_1\in\Gamma(T^*\tilde{M_1})$ and $\omega_2,\eta_2\in\Gamma(T^*\tilde{M_2})$, we have
	\begin{align*}
	\textbf{(i)}.\:\mathcal{D}_{\omega_1^h}\beta_1^h&=(\mathcal{D}_{\omega_1}^1\eta_1)^h,\\
	\textbf{(ii)}.\:\mathcal{D}_{\omega_2^v}\eta_2^v&=(\mathcal{D}_{\omega_2}^2\eta_2)^v-\frac{1}{(f^h)^3}g_2(\omega_2,\eta_2)^v(J_1df)^h,\\
	\textbf{(iii)}.\:\mathcal{D}_{\omega_1^h}\eta_2^v&=\mathcal{D}_{\eta_2^v}\omega_1^h=\frac{1}{f^h}g_1(J_1df,\omega_1)^h\eta_2^v.
	\end{align*}
\end{proposition}
If we assume that, $\Pi=\Pi_1+\Pi_2$ in Theorem 5.2 of (\cite{yar},p. 294) then we conclude that:
\begin{theorem}
	Let $f$ be a Casimir function. Then both $(\tilde{M_1},g_1,\Pi_1)$ and $(\tilde{M_2},g_2,\Pi_2)$ are Riemannian Poisson manifolds if and only if the triplet $(\tilde{M}=\tilde{M_1}\times_f \tilde{M_2},g^f,\Pi)$ is a Riemannian Poisson warped product space.
\end{theorem}
\section{Killing 1-form}
Let $(\tilde{M},g,\Pi)$ be a Riemannian Poisson manifold. In ( \cite{rfl},p. 5), author define the Lie derivative on the space of $k$-vector fields $\mathfrak{X}^k(\tilde{M})=\Gamma(\wedge^kT\tilde{M})$. Let $T\in\mathfrak{X}^k(\tilde{M})$, then the Lie derivative of $T$ in the direction of 1-form $\alpha\in\Omega^1(\tilde{M})$ is a map $\mathcal{L}_\alpha:\mathfrak{X}^k(\tilde{M})\rightarrow\mathfrak{X}^k(\tilde{M})$ such that
\begin{equation}
(\mathcal{L}_\alpha T)(\alpha_1,...,\alpha_k)=\sharp_\Pi(\alpha)(T(\alpha_1,...,\alpha_k))-\sum_{i=1}^{k}T(\alpha_1,...,[\alpha,\alpha_i]_\Pi,...,\alpha_k),
\end{equation} 
where $\alpha_1,...,\alpha_k\in\Omega^1(M)$.\par 
A 1-form $\eta\in\Omega^1(\tilde{M})$ on $(\tilde{M},g,\Pi)$ is said to be Killing 1-form corresponding to the cometric $g$ if 
$$\mathcal{L}_\eta g=0,$$
where $\mathcal{L}_\eta$ is Lie derivative on $\tilde{M}$ with respect to 1-form $\eta$.\par
In the following two propositions, we will find the expression of Lie derivative $\mathcal{L}_\eta$ with respect to cometric $g^f$ on contravariant warped product space $(\tilde{M}=\tilde{M_1}\times_f \tilde{M_2},g^f)$ and Riemannian Poisson warped product space $(M=\tilde{M_1}\times_f \tilde{M_2},g^f,\Pi)$. Here we will consider $\eta=\eta_1^h+\eta_2^v$, $\alpha=\alpha_1^h+\alpha_2^v$ and $\beta=\beta_1^h+\beta_2^v$.
\begin{proposition} Let $(\tilde{M}=\tilde{M_1}\times_f \tilde{M_2},g^f)$ be a contravariant warped product space and $\mathcal{D}$ is the contravariant Levi-Civita connection associated with pair $(g^f,\Pi)$ on $\tilde{M}$. Then for any $\eta\in\Omega^1(\tilde{M})$, we have
	\begin{align*}
	(\mathcal{L}_{\eta}g^f)(\alpha,\beta)&=\big[(\mathcal{L}_{\eta_1}^1g_1)(\alpha_1,\beta_1)\big]^h\\
	&+\frac{1}{(f^h)^2}\big[(\mathcal{L}_{\eta_2}^2g_2)(\alpha_2,\beta_2)\big]^v+\Big(\frac{g_1(J_1df,\eta_1)}{f^3}\Big)^hg_2(\alpha_2,\beta_2)^v,
	\end{align*}	
	for any $\alpha,\beta\in\Omega^1(\tilde{M})$.
\end{proposition}
\begin{proof}
From equation $(3.1)$, we conclude that
\begin{align*}
(\mathcal{L}_{\eta}g^f)(\alpha,\beta)&=\sharp_\Pi(\eta)(g^f(\alpha,\beta))-g^f([\eta,\alpha]_\Pi,\beta)-g^f(\alpha,[\eta,\beta]_\Pi)\\
&=\big[\sharp_{\Pi_1}(\eta_1)(g_1(\alpha_1,\beta_1))\big]^h+\big[\sharp_{\Pi_1}(\eta_1)\big]^h(\frac{1}{(f^h)^2}g_2(\alpha_2,\beta_2)^v)\\
&+\frac{1}{(f^h)^2}\big[\sharp_{\Pi_2}(\eta_2)(g_2(\alpha_2,\beta_2))\big]^v+\big[g_1([\eta_1,\alpha_1]_{\Pi_1},\beta_1)\big]^h\\
&+\frac{1}{(f^h)^2}\big[g_2([\eta_2,\alpha_2]_{\Pi_2},\beta_2)\big]^v+\big[g_1(\alpha_1,[\eta_1,\beta_1]_{\Pi_1})\big]^h\\
&+\frac{1}{(f^h)^2}\big[g_2(\alpha_2,[\eta_2,\beta_2]_{\Pi_2})\big]^v\\
&=\big[(\mathcal{L}_{\eta_1}^1g_1)(\alpha_1,\beta_1)\big]^h\\
&+\frac{1}{(f^h)^2}\big[(\mathcal{L}_{\eta_2}^2g_2)(\alpha_2,\beta_2)\big]^v+\big[\sharp_{\Pi_1}(\eta_1)\big]^h(\frac{1}{(f^h)^2}g_2(\alpha_2,\beta_2)^v).
\end{align*}
Since,
\begin{equation*}
\big[\sharp_{\Pi_1}(\eta_1)\big]^h(\frac{1}{(f^h)^2}g_2(\alpha_2,\beta_2)^v)=\Big(\frac{g_1(J_1df,\eta_1)}{f^3}\Big)^hg_2(\alpha_2,\beta_2)^v.
\end{equation*}
Thus the result follows.
\end{proof}
\begin{proposition}
	Let $(\tilde{M}=\tilde{M_1}\times_f\tilde{M_2},g^f,\Pi)$ be a Riemannian Poisson warped product space and $f$ is a Casimir function on $\tilde{M_1}$. Then for any $\eta\in\Omega^1(\tilde{M})$, we have
	\begin{equation*}
	(\mathcal{L}_{\eta}g^f)(\alpha,\beta)=\big[(\mathcal{L}_{\eta_1}^1g_1)(\alpha_1,\beta_1)\big]^h+\frac{1}{(f^h)^2}\big[(\mathcal{L}_{\eta_2}^2g_2)(\alpha_2,\beta_2)\big]^v,
	\end{equation*}	
	for any $\alpha,\beta\in\Omega^1(\tilde{M})$.
\end{proposition}
\begin{proof}
As, $f$ is Casimir function if and only if $J_1df=0$. After applying this criterion in Proposition 3.1 provides the result.
\end{proof}
The following proposition is a another characterization of Killing 1-form.
\begin{proposition}
Let $(\tilde{M},g,\Pi)$ be a Riemannian Poisson manifold. A 1-form $\eta\in\Omega^1(\tilde{M})$ is a Killing 1-form if and only if 
\begin{equation}
g(\mathcal{D}_\alpha\eta,\alpha)=0,
\end{equation}
for any 1-form $\alpha\in\Omega^1(\tilde{M})$.
\end{proposition}
\begin{proof}
Since $\eta\in\Omega^1(\tilde{M})$ and $\mathcal{D}$ is the contravariant Levi-Civita connection, then
	\begin{equation}
	(\mathcal{L}_\eta g)(\alpha,\beta)=g(\mathcal{D}_\alpha\eta,\beta)+g(\alpha,\mathcal{D}_\beta\eta),
	\end{equation}
	for any $\alpha,\beta\in\Omega^1(\tilde{M})$.
Putting $\alpha=\beta$ in $(3.3)$, we have
\begin{equation*}
(\mathcal{L}_\eta g)(\alpha,\alpha)=2g(\mathcal{D}_\alpha\eta,\alpha),
\end{equation*}
for any $\alpha\in\Omega^1(M)$. Thus the result follows.
\end{proof}
 In the preceding two propositions, we will provide a result on contravariant warped product space $(\tilde{M}=\tilde{M_1}\times_f \tilde{M_2},g^f)$ and Riemannian Poisson warped product space $(\tilde{M}=\tilde{M_1}\times_f\tilde{M_2},g^f,\Pi)$, which are helpful to describe the Killing 1-form. Here we will consider $\eta=\eta_1^h+\eta_2^v\:and\:\alpha=\alpha_1^h+\alpha_2^v$.
\begin{proposition} Let $(\tilde{M}=\tilde{M_1}\times_f \tilde{M_2},g^f)$ be a contravariant warped product space and $\mathcal{D}$ is the contravariant Levi-Civita connection associated with pair $(g^f,\Pi)$ on $\tilde{M}$. Then for any $\eta,\alpha\in\Omega^1(\tilde{M})$, we have
\begin{equation*}
g^f(\mathcal{D}_\alpha\eta,\alpha)=g_1(\mathcal{D}_{\alpha_1}^1\eta_1,\alpha_1)^h+\Big(\frac{g_1(J_1df,\eta_1)}{f^3}\Big)^h(||\alpha_2||_2^2)^v+\frac{1}{(f^h)^2}g_2(\mathcal{D}_{\alpha_2}^2\eta_2,\alpha_2)^v.
\end{equation*}	 
\end{proposition}
\begin{proof}
From Proposition 2.2, for any $\eta,\alpha\in\Omega^1(M)$, we have
\begin{align}
\mathcal{D}_\alpha\eta&=(\mathcal{D}_{\alpha_1}^1\eta_1)^h+\Big(\frac{g_1(J_1df,\alpha_1)}{f}\Big)^h\eta_2^v+\Big(\frac{g_1(J_1df,\eta_1)}{f}\Big)^h\alpha_2^v\nonumber\\
&-\frac{1}{(f^h)^3}g_2(\alpha_2,\eta_2)^v(J_1df)^h+(\mathcal{D}_{\alpha_2}^2\eta_2)^v.
\end{align}
Since $g^f(\mathcal{D}_\alpha\eta,\alpha)=g^f(\mathcal{D}_\alpha\eta,\alpha_1^h)+g^f(\mathcal{D}_\alpha\eta,\alpha_2^v)$, thus from $(2.5)$ and $(3.4)$, the result follows.
\end{proof}
\begin{proposition}
Let $(\tilde{M}=\tilde{M_1}\times_f\tilde{M_2},g^f,\Pi)$ be a Riemannian Poisson warped product space and $f$ is a Casimir function on $\tilde{M_1}$. Then for any $\eta,\alpha\in\Omega^1(M)$, we have
\begin{equation}
g^f(\mathcal{D}_\alpha\eta,\alpha)=g_1(\mathcal{D}_{\alpha_1}^1\eta_1,\alpha_1)^h+\frac{1}{(f^h)^2}g_2(\mathcal{D}_{\alpha_2}^2\eta_2,\alpha_2)^v.
\end{equation}
\end{proposition}
\begin{proof}
As, $f$ is Casimir function if and only if $J_1df=0$. After applying this criterion in Proposition 3.4, provides the result.
\end{proof}
In the following theorem, we have to prove the necessary and sufficient conditions for Killing 1-form on Riemannian Poisson warped product space.
\begin{theorem}
Let $(\tilde{M}=\tilde{M_1}\times_f\tilde{M_2},g^f,\Pi)$ be a Riemannian Poisson warped product space and $f$ is a Casimir function on $\tilde{M_1}$. Then 1-form $\eta\in\Omega^1(M)$ is Killing 1-form if and only if the following conditions holds:\\
\textbf{(1)}.$\:\eta_1$ is a Killing 1-form on $\tilde{M_1}$.\\
\textbf{(2)}.$\:\eta_2$ is a Killing 1-form on $\tilde{M_2}$.
\end{theorem}
\begin{proof}
The if" part is obvious. For the "only if part", let $\eta\in\Omega^1(\tilde{M})$ is Killing 1-form. Putting $\eta=\eta_1^h$ and $\eta=\eta_2^v$ in $(3.5)$ provide $(1)$ and $(2)$ respectively.
\end{proof}
\subsection{Parallel 1-form}
Let $(\tilde{M}^n,\Pi)$ be the n-dimensional Poisson manifold and $\mathcal{D}$ is the contravariant Levi-Civita connection associated to $(\Pi,g)$, then\\
\textbf{(i)} In (\cite{yar},eqn. $5$), authors provided contravariant derivative of a multivector field $P$ of degree $r$ i.e., $P\in\mathfrak{X}^r(\tilde{M})=\Gamma(\Lambda^rT\tilde{M})$ with respect to $1$-form $\alpha\in\Omega^1(\tilde{M})$, given by 
\begin{equation}
(\mathcal{D}_{\alpha}P)(\alpha_1,...,\alpha_r)=\sharp_{\Pi}(\alpha)(P(\alpha_1,...,\alpha_r))-\displaystyle\sum_{i=1}^{r}
Q\big(\alpha_1,...,\mathcal{D}_{\alpha}\alpha_i,...,\alpha_r\big),
\end{equation}
where $\alpha_1,...,\alpha_r\in\Omega^1(\tilde{M})$.\\

\textbf{(ii)} Let $Q$ be any tensor field on $\tilde{M}$. In (\cite{zs},p. 9), author provided contravariant Laplacian operator corresponding  to $\mathcal{D}$ over $Q$  by
\begin{equation}
\Delta^\mathcal{D}(Q):=-\displaystyle\sum_{i=1}^{n}\mathcal{D}^2_{\theta_i,\theta_i}Q,
\end{equation}
where $\{\theta_1,...,\theta_n\}$ is any local coframe field on $\tilde{M}$, and
\begin{equation*}
\mathcal{D}^2_{\alpha,\beta}=\mathcal{D}_\alpha\mathcal{D}_\beta-\mathcal{D}_{\mathcal{D}_\alpha\beta}
\end{equation*}
is the second order contravariant derivative.
\begin{definition}
Let $\mathcal{D}$ is the contravariant Levi-Civita connection associated to $(\Pi,g)$ on Poisson manifold $(\tilde{M},\Pi)$. A tensor field $S$ is said to be parallel with respect to contravariant Levi-Civita connection $\mathcal{D}$ if 
\begin{equation}
\mathcal{D}S=0.
\end{equation}
\end{definition}
\begin{remark}
If we take $S=g$, then it is always parallel.
\end{remark}
From Corollary 4.2, Lemma 4.3 and Corollary 4.7 of \cite{zs}, we conclude the following lemma. This will be useful later on.
\begin{lemma}
	Let $(\tilde{M}^n,g,\Pi)$ be a compact Riemannian Poisson manifold and a smooth function $f$ on $\tilde{M}$ satisfies $\Delta^\mathcal{D}(f)\geq0$, then $\Delta^\mathcal{D}(f)=0$. 
\end{lemma} 
Bochner \cite{sbo}, provided a result for compact oriented Riemannian manifold $\tilde{M}$, that if Ricci curvature of $\tilde{M}$ is non-positive then every Killing vector field on $\tilde{M}$ is parallel. Later H. H. Wu studied this result in detail (see, \cite{hhw},p. 324). Now we will prove similar result for Killing 1-form on compact Riemannian Poisson manifold.
\begin{theorem}
Let $\eta$ is a Killing 1-form on $n$-dimensional compact Riemannian Poisson manifold $(\tilde{M}^n,g,\Pi)$ with vanishing $\mathcal{D}_\eta\eta$. If $Ric(\eta,\eta)\leq0$, then $\eta$ is parallel.
\end{theorem} 
\begin{proof}
Since $\eta$ is a Killing 1-form, equation $(3.3)$, implies that
\begin{equation}
g(\mathcal{D}_\alpha\eta,\beta)+g(\mathcal{D}_\beta\eta,\alpha)=0,
\end{equation}
for any $\alpha,\beta\in\Omega^1(\tilde{M})$.
Let $\{\theta_1,...,\theta_n\}$ is any local coframe field on $M$, then from $(3.7)$, we have
\begin{align}
\Delta^\mathcal{D}\big(-\frac{1}{2}|\eta|^2\big)&=\sum_{i=1}^{n}\{\mathcal{D}_{\theta_i}(\mathcal{D}_{\theta_i}g(\eta,\eta))-\mathcal{D}_{\mathcal{D}_{\theta_i}\theta_i}(g(\eta,\eta))\}\nonumber\\
&=\sum_{i=1}^{n}\{\mathcal{D}_{\theta_i}(g(\mathcal{D}_{\theta_i}\eta,\eta))-g(\mathcal{D}_{\mathcal{D}_{\theta_i}\theta_i}\eta,\eta)\}\nonumber\\
&=|\mathcal{D}\eta|^2-g(\Delta^\mathcal{D}(\eta),\eta),
\end{align} 
where $|\mathcal{D}\eta|^2=\sum_{i=1}^{n}g(\mathcal{D}_{\theta_i}\eta,\mathcal{D}_{\theta_i}\eta)$. Now we will calculate the second term of $(3.10)$. For any $i\in\{1,....,n\}$, we have
\begin{equation}
g(\mathcal{D}_{\theta_i,\theta_i}^2\eta,\eta)=g(\mathcal{D}_{\theta_i}\mathcal{D}_{\theta_i}\eta,\eta)-g(\mathcal{D}_{\mathcal{D}_{\theta_i}\theta_i}\eta,\eta).
\end{equation} 
The second term of L. H. S. of the above equation equal to $-g(\mathcal{D}_\eta\eta,\mathcal{D}_{\theta_i}\theta_i)$ by $(3.9)$, and vanishes as $\mathcal{D}_\eta\eta$. Hence $(3.11)$, conclude that
\begin{align}
g(\mathcal{D}_{\theta_i,\theta_i}^2\eta,\eta)&=g(\mathcal{D}_{\theta_i}\mathcal{D}_{\theta_i}\eta,\eta)\nonumber\\
&=g(\mathcal{D}_{\theta_i}\mathcal{D}_\eta\theta_i,\eta)+g(\mathcal{D}_{\theta_i}[\theta_i,\eta]_\Pi,\eta)\nonumber\\
&=g(\mathcal{D}_{\theta_i}\mathcal{D}_\eta\theta_i,\eta)+\sharp_\Pi(\theta_i)g([\theta_i,\eta]_\Pi,\eta)-g([\theta_i,\eta]_\Pi,\mathcal{D}_{\theta_i}\eta).
\end{align}
Since, $\alpha$ is Killing 1-form therefore 
\begin{equation}
g([\theta_i,\eta]_\Pi,\eta)=-\sharp_\Pi(\eta)g(\theta_i,\eta)
\end{equation}
and 
\begin{align}
g([\theta_i,\eta]_\Pi,\mathcal{D}_{\theta_i}\eta)&\stackrel{(3.9)}{=}-g(\theta_i,\mathcal{D}_{[\theta_i,\eta]_\Pi}\eta)\nonumber\\
&=-\sharp_\Pi([\theta_i,\eta]_\Pi)g(\eta,\theta_i)+g(\eta,\mathcal{D}_{[\theta_i,\eta]_\Pi}\theta_i).
\end{align}
After using $(3.13)$ and $(3.14)$ in $(3.12)$, we obtain
\begin{align}
g(\mathcal{D}_{\theta_i,\theta_i}^2\eta,\eta)&=g(\mathcal{D}_{\theta_i}\mathcal{D}_\eta\theta_i,\eta)+\{-\sharp_\Pi(\theta_i)\sharp_\Pi(\eta)+\sharp_\Pi([\theta_i,\eta]_\Pi)\}g(\eta,\theta_i)\nonumber\\
&-g(\mathcal{D}_{[\theta_i,\eta]_\Pi}\theta_i,\eta)\nonumber\\
&\stackrel{(2.1)}{=}g(\mathcal{D}_{\theta_i}\mathcal{D}_\eta\theta_i,\eta)-\sharp_\Pi(\eta)\sharp_\Pi(\theta_i)g(\eta,\theta_i)-g(\mathcal{D}_{[\theta_i,\eta]_\Pi}\theta_i,\eta).
\end{align} 
The second term of $(3.15)$ follows by vanishing of $\mathcal{D}_\eta\eta$,
\begin{align}
\sharp_\Pi(\eta)\sharp_\Pi(\theta_i)g(\eta,\theta_i)&=\sharp_\Pi(\eta)\{g(\mathcal{D}_{\theta_i}\eta,\theta_i)+g(\eta,\mathcal{D}_{\theta_i}\theta_i)\}\nonumber\\
&\stackrel{(3.9)}{=}\sharp_\Pi(\eta)g(\eta,\mathcal{D}_{\theta_i}\theta_i)\nonumber\\
&=g(\mathcal{D}_\eta\mathcal{D}_{\theta_i}\theta_i,\eta).
\end{align}
Using equation $(3.16)$ in $(3.15)$, yields
\begin{equation*}
g(\mathcal{D}_{\theta_i,\eta_i}^2\eta,\eta)=g(\mathcal{R}(\theta_i,\eta)\theta_i,\eta).
\end{equation*}
 After taking summation both sides of the above equation conclude that 
\begin{equation}
g(\Delta^\mathcal{D}(\eta),\eta)=Ric(\eta,\eta).
\end{equation}
Now using $(3.17)$ in $(3.10)$, we have
\begin{align}
\Delta^\mathcal{D}\big(-\frac{1}{2}|\eta|^2\big)&=|\mathcal{D}\eta|^2-Ric(\eta,\eta).
\end{align} 
Since $Ric(\eta,\eta)\leq0$ then the right hand side of $(3.18)$ is non-negative and hence vanishes by Lemma 3.9. It conclude that $|\mathcal{D}\eta|^2=0$. This is equivalent to $\eta$ being parallel. 
\end{proof}
In the following theorem, we will prove the above result for compact Riemannian Poisson warped product space.
\begin{theorem}
	Let $(\tilde{M}=\tilde{M_1}\times_f\tilde{M_2},g^f,\Pi)$ be a compact Riemannian Poisson warped product space and $f$ is a Casimir function on $\tilde{M_1}$ also let 1-form $\eta=\eta_1^h+\eta_2^v\in\Omega^1(\tilde{M})$. Then\\
	\textbf{(1)}. $\eta=\eta_1^h+\eta_2^v$ is parallel if the 1-form $\eta_i$ is a Killing 1-form, $Ric_i(\eta_i,\eta_i)\leq0$ and $\mathcal{D}_{\eta_i}^i\eta_i$ vanishes, $i=1,2$.\\
	\textbf{(2)}. $\eta=\eta_1^h$ is parallel if the 1-form $\eta_1$ is a Killing 1-form, $Ric_1(\eta_1,\eta_1)\leq0$ and $\mathcal{D}_{\eta_1}^1\eta_1$ vanishes.\\
	\textbf{(3)}. $\eta=\eta_2^v$ is parallel if the 1-form $\eta_2$ is a Killing 1-form, $Ric_2(\eta_2,\eta_2)\leq0$ and $\mathcal{D}_{\eta_2}^2\eta_2$ vanishes.
\end{theorem}
\begin{proof}
Let $U_1$ and $U_2$ are two open subset of $\tilde{M_1}$ and $\tilde{M_2}$ respectively. Assume that $\{dx_1,...,dx_{n_1}\}$ is a local $g_1$-coframe on $U_1$ and $\{dy_1,...,dy_{n_2}\}$ is a local $g_2$-coframe on $U_2$, then
\begin{equation*}
\{dx_1^{h},...,dx_{n_1}^{h},f^hdy_1^{v},...,f^hdy_{n_2}^{v}\}
\end{equation*}
is a local $g^f$-coframe on open subset $U_1\times U_2$ of $\tilde{M_1}\times \tilde{M_2}$.
Thus for any 1-forms $\eta\in\Omega^1(\tilde{M})$, we have
\begin{equation}
|\mathcal{D}\eta|^2=\sum_{i=1}^{n_1}g^f(\mathcal{D}_{dx_i^h}\eta,\mathcal{D}_{dx_i^h}\eta)+(f^h)^2\sum_{j=1}^{n_2}g^f(\mathcal{D}_{dy_j^v}\eta,\mathcal{D}_{dy_j^v}\eta).
\end{equation}
Using the condition of Casimir function $f$ in $(3.4)$ the first term of $(3.19)$, is given by
\begin{align}
\sum_{i=1}^{n_1}g^f(\mathcal{D}_{dx_i^h}\eta,\mathcal{D}_{dx_i^h}\eta)&=\sum_{i=1}^{n_1}g^f(\mathcal{D}_{dx_i^h}\eta_1^h,\mathcal{D}_{dx_i^h}\eta_1^h)\nonumber\\
&=\sum_{i=1}^{n_1}g^f((\mathcal{D}_{dx_i}^1\eta_1)^h,(\mathcal{D}_{dx_i}^1\eta_1)^h)\nonumber\\
&\stackrel{(2.5)}{=}\sum_{i=1}^{n_1}g_1(\mathcal{D}_{dx_i}^1\eta_1,\mathcal{D}_{dx_i}^1\eta_1)^h\nonumber\\
&=(|\mathcal{D}^1\eta_1|^2)^h,
\end{align}
and the second term of $(3.19)$, is given by
\begin{align}
(f^h)^2\sum_{j=1}^{n_2}g^f(\mathcal{D}_{dy_j^v}\eta,\mathcal{D}_{dy_j^v}\eta)&=(f^h)^2\sum_{j=1}^{n_2}g^f(\mathcal{D}_{dy_j^v}\eta_2^v,\mathcal{D}_{dy_j^v}\eta_2^v)\nonumber\\
&=(f^h)^2\sum_{j=1}^{n_2}g^f((\mathcal{D}_{dy_j}^2\eta_2)^v,(\mathcal{D}_{dy_j}^2\eta_2)^v)\nonumber\\
&\stackrel{(2.5)}{=}\sum_{j=1}^{n_2}g_2(\mathcal{D}_{dy_j}^2\eta_2,\mathcal{D}_{dy_j}^2\eta_2)^v\nonumber\\
&=(|\mathcal{D}^2\eta_2|^2)^v,
\end{align}
After using $(3.20)$ and $(3.21)$ in $(3.19)$, provide that 
\begin{equation}
|\mathcal{D}\eta|^2=(|\mathcal{D}^1\eta_1|^2)^h+(|\mathcal{D}^2\eta_2|^2)^v.
\end{equation}
Thus from Theorem 3.10 and equation $(3.22)$, follows the result.
\end{proof}
\section{2-Killing 1-form}
A 1-form $\eta\in\Omega^1(\tilde{M})$ on a Riemannian Poisson manifold $(\tilde{M},g,\Pi)$ is said to be 2-Killing 1-form with corresponding to the metric $g$ if
\begin{equation}
\mathcal{L}_\eta\mathcal{L}_\eta g=0,
\end{equation} 
where $\mathcal{L}_\eta$ is the Lie derivative on $\tilde{M}$ corresponding to 1-form $\eta$.\\ 
The following proposition is alike to the Proposition 3.1 of (\cite{ssb},p. 6).
\begin{proposition}
Let $(\tilde{M},g,\Pi)$ be a Riemannian Poisson manifold and 1-form $\eta\in\Omega^1(\tilde{M})$. Then
\begin{align}
(\mathcal{L}_\eta\mathcal{L}_\eta g)(\alpha,\beta)&=g(\mathcal{D}_\eta\mathcal{D}_\alpha\eta-\mathcal{D}_{[\eta,\alpha]_\Pi}\eta,\beta)\nonumber\\
&+g(\mathcal{D}_\eta\mathcal{D}_\beta\eta-\mathcal{D}_{[\eta,\beta]_\Pi}\eta,\alpha)+2g(\mathcal{D}_\alpha\eta,\mathcal{D}_\beta\eta),	
\end{align}
for any $\alpha,\beta\in\Omega^1(\tilde{M})$. 
\end{proposition}
The following proposition is helpful to describe the definition of 2-Killing 1-form on the Riemannian Poisson manifold.
\begin{proposition}
Let $(\tilde{M},g,\Pi)$ be a Riemannian Poisson manifold and 1-form $\eta\in\Omega^1(\tilde{M})$. Then $\eta$ is 2-Killing 1-form if and only if
\begin{equation}
\mathcal{R}(\eta,\alpha,\alpha,\eta)=g(\mathcal{D}_\alpha\eta,\mathcal{D}_\alpha\eta)+g(\mathcal{D}_\alpha\mathcal{D}_\eta\eta,\alpha),
\end{equation}
for any $\alpha\in\Omega^1(\tilde{M})$.
\end{proposition}
\begin{proof}
The symmetry of $(4.2)$ implies that, $\eta$ is 2-Killing 1-form if and only if $(\mathcal{L}_\eta\mathcal{L}_\eta g)(\alpha,\alpha)=0$, for any $\alpha\in\Omega^1(\tilde{M})$. Therefore, we have
\begin{equation}
g(\mathcal{D}_\eta\mathcal{D}_\alpha\eta,\alpha)+g(\mathcal{D}_{[\alpha,\eta]_\Pi}\eta,\alpha)+g(\mathcal{D}_\alpha\eta,\mathcal{D}_\alpha\eta)=0,
\end{equation}
for any $\alpha\in\Omega^1(\tilde{M})$.
The curvature tensor $\mathcal{R}$, is given by
\begin{align}
\mathcal{R}(\eta,\alpha,\alpha,\eta)&=\mathcal{R}(\alpha,\eta,\eta,\alpha)\nonumber\\
&=g(\mathcal{R}(\alpha,\eta)\eta,\alpha)\nonumber\\
&=g(\mathcal{D}_\alpha\mathcal{D}_\eta\eta,\alpha)-g(\mathcal{D}_\eta\mathcal{D}_\alpha\eta,\alpha)-g(\mathcal{D}_{[\alpha,\eta]_\Pi}\eta,\alpha).
\end{align}
After using $(4.4)$ in $(4.5),$ provides the result $(4.3)$.
\end{proof}
There is another characterization for a 2-Killing 1-form $\eta$ on Riemannian Poisson manifold $(\tilde{M},g,\Pi)$
\begin{equation}
2\mathcal{R}(\eta,\alpha,\beta,\eta)=2g(\mathcal{D}_\alpha\eta,\mathcal{D}_\beta\eta)+g(\mathcal{D}_\alpha\mathcal{D}_\eta\eta,\beta)+g(\mathcal{D}_\beta\mathcal{D}_\eta\eta,\alpha),
\end{equation} 
for any $\alpha,\beta\in\Omega^1(\tilde{M})$.\par
In the following two theorems, we will provide Bochner-type results for 2-Killing 1-form on compact Riemannian Poisson manifold and compact Riemannian Poisson warped product space.
\begin{theorem}
Let $\eta$ is a 2-Killing 1-form on $n$-dimensional compact Riemannian Poisson manifold $(\tilde{M},g,\Pi)$ with vanishing $\mathcal{D}_\eta\eta$. If $Ric(\eta,\eta)\leq0$, then $\eta$ is parallel. 
\end{theorem}
\begin{proof}
Assume that $\{dx_1,...,dx_{n}\}$ is a local $g$-coframe on an open subset $U$ of $\tilde{M}$, then from Proposition 4.2, we obtain
\begin{equation*}
\sum_{i=1}^{n}\mathcal{R}(\eta,dx_i,dx_i,\eta)=\sum_{i=1}^{n}g(\mathcal{D}_{dx_i}\eta,\mathcal{D}_{dx_i}\eta)+\sum_{i=1}^{n}g(\mathcal{D}_{dx_i}\mathcal{D}_\eta\eta,dx_i).
\end{equation*}
As $\mathcal{D}_\eta\eta$ vanishes and $\mathcal{R}$ is a curvature tensor therefore the last equation implies that 
\begin{equation*}
Ric(\eta,\eta)=|\mathcal{D}\eta|^2\leq0.
\end{equation*}
This follows the result.
\end{proof}
\begin{theorem}
	Let $(M=\tilde{M_1}\times_f\tilde{M_2},g^f,\Pi)$ be a compact Riemannian Poisson warped product space and $f$ is a Casimir function on $B$ also let 1-form $\eta=\eta_1^h+\eta_2^v\in\Omega^1(\tilde{M})$. Then\\
	\textbf{(1)}. $\eta=\eta_1^h+\eta_2^v$ is parallel if the 1-form $\eta_i$ is a 2-Killing 1-form, $Ric_i(\eta_i,\eta_i)\leq0$ and $\mathcal{D}_{\eta_i}^i\eta_i$ vanishes, $i=1,2$.\\
	\textbf{(2)}. $\eta=\eta_1^h$ is parallel if the 1-form $\eta_1$ is a 2-Killing 1-form, $Ric_1(\eta_1,\eta_1)\leq0$ and $\mathcal{D}_{\eta_1}^1\eta_1$ vanishes.\\
	\textbf{(3)}. $\eta=\eta_2^v$ is parallel if the 1-form $\eta_2$ is a 2-Killing 1-form, $Ric_2(\eta_2,\eta_2)\leq0$ and $\mathcal{D}_{\eta_2}^2\eta_2$ vanishes.
\end{theorem}
\begin{proof}
Proof is similar to the Theorem 3.11.
\end{proof}
In the following two propositions, we will find the expression for 2-Killing 1-form.
\begin{proposition}
Let $(\tilde{M}=\tilde{M_1}\times_f \tilde{M_2},g^f)$ be a contravariant warped product space and $\mathcal{D}$ is the contravariant Levi-Civita connection associated with pair $(g^f,\Pi)$ on $\tilde{M}$. Then for any 1-forms $\eta\in\Omega^1(\tilde{M})$, we have
\begin{align*}
(\mathcal{L}_\eta\mathcal{L}_\eta g^f)&(\alpha,\beta)=\big[(\mathcal{L}_{\eta_1}^1\mathcal{L}_{\eta_1}^1 g_1)(\alpha_1,\beta_1)\big]^h+\frac{1}{(f^h)^2}\big[(\mathcal{L}_{\eta_2}^2\mathcal{L}_{\eta_2}^2 g_2)(\alpha_2,\beta_2)\big]^v\\
&+2\Big(\mathcal{D}_{\eta_1}^1(\frac{g_1(J_1df,\eta_1)}{f^3})+\frac{2g_1(J_1df,\eta_1)^2}{f^4}\Big)^hg_2(\alpha_2,\beta_2)^v\\
&+2\Big(\frac{\mathcal{D}_{\eta_1}^1(f)g_1(J_1df,\beta_1)}{f^4}+\frac{g_1(J_1df,\beta_1)g_1(J_1df,\eta_1)}{f^4}\Big)^hg_2(\alpha_2,\eta_2)^v\\
&+2\Big(\frac{\mathcal{D}_{\eta_1}^1(f)g_1(J_1df,\alpha_1)}{f^4}+\frac{g_1(J_1df,\alpha_1)g_1(J_1df,\eta_1)}{f^4}\Big)^hg_2(\beta_2,\eta_2)^v\\
&+4\Big(\frac{g_1(J_1df,\eta_1)}{f^3}\Big)^h\big((\mathcal{L}_{\eta_2}^2g_2)(\alpha_2,\beta_2)\big)^v+2\Big(\frac{g_1(J_1df,\alpha_1)}{f^3}\Big)^hg_2(\eta_2,\mathcal{D}_{\beta_2}^2\eta_2)^v\\
&+2\Big(\frac{g_1(J_1df,\beta_1)}{f^3}\Big)^hg_2(\eta_2,\mathcal{D}_{\alpha_2}^2\eta_2)^v+4\Big(\frac{g_1(J_1df,\alpha_1)g_1(J_1df,\beta_1)}{f^4}\Big)^h(||\eta_2||_2^2)^v,
\end{align*}
for any 1-forms $\alpha,\beta\in\Omega^1(\tilde{M})$.
\end{proposition}
\begin{proof}
See Appendix.
\end{proof}
\begin{proposition}
	Let $(\tilde{M}=\tilde{M_1}\times_f\tilde{M_2},g^f,\Pi)$ be a Riemannian Poisson warped product space and $f$ is a Casimir function on $\tilde{M_1}$. Then for any 1-forms $\eta\in\Omega^1(\tilde{M})$, we have
	\begin{equation*}
	(\mathcal{L}_\eta\mathcal{L}_\eta g^f)(\alpha,\beta)=\big[(\mathcal{L}_{\eta_1}^1\mathcal{L}_{\eta_1}^1 g_1)(\alpha_1,\beta_1)\big]^h+\frac{1}{(f^h)^2}\big[(\mathcal{L}_{\eta_2}^2\mathcal{L}_{\eta_2}^2 g_2)(\alpha_2,\beta_2)\big]^v,\\
	\end{equation*}
	for any $\alpha,\beta\in\Omega^1(M)$.
\end{proposition}
\begin{proof}
Using the property of Casimir function $f$ in Proposition 4.5, provides this result.
\end{proof}
In the following theorem, we will provide necessary and sufficient conditions for 2-Killing 1-form on Riemannian Poisson warped product space.
\begin{theorem}
	Let $(\tilde{M}=\tilde{M_1}\times_f\tilde{M_2},g^f,\Pi)$ be a Riemannian Poisson warped product space and $f$ is a Casimir function on $\tilde{M_1}$. Then 1-form $\eta\in\Omega^1(\tilde{M})$ is 2-Killing 1-form if and only if the following conditions holds:\\
	\textbf{(1)}.$\:\eta_1$ is a 2-Killing 1-form on $\tilde{M_1}$.\\
	\textbf{(2)}.$\:\eta_2$ is a 2-Killing 1-form on $\tilde{M_2}$.
\end{theorem}
\begin{proof}
	The if" part is obvious. For the "only if part", let $\eta\in\Omega^1(\tilde{M})$ is 2-Killing 1-form. Putting $\eta=\eta_1^h$ and $\eta=\eta_2^v$ in Proposition 4.6 provide $(1)$ and $(2)$ respectively.
\end{proof}
Now, we will provide a theorem for 2-Killing 1-form. From (\cite{rfl},eqn. 2.5), Christoffel symbols $\Gamma_k^{ij}$ defined as
\begin{equation}
\mathcal{D}_{dx_i}dx_j=\Gamma_k^{ij}dx_k.
\end{equation}
\begin{theorem}
Let $(\mathbb{R}^2,g,\Pi)$ be a Riemannian Poisson manifold (where g is the cometric of the Riemannian metric $\tilde{g}=(dx^1)^2+(dx^2)^2$, $\Pi=\Pi^{12}\frac{\partial}{\partial x^1}\wedge\frac{\partial}{\partial x^2}$) and $\eta=\eta_1dx^1+\eta_2dx^2\in\Omega^1(\mathbb{R}^2)$. Then $\eta$ is 2- Killing form if and only if 
\begin{equation*}
2\mathcal{R}(\eta,dx^1,dx^2,\eta)=-\bigg(2(T_1T_3+T_2T_4)+\frac{\partial(T_5\Pi^{12})}{\partial x^1}+\frac{\partial(T_6\Pi^{12})}{\partial x^2}\bigg), 
\end{equation*} 
where
\begin{align*}
T_1&=\Pi^{12}\frac{\partial\eta_1}{\partial x^2}+\eta_2\frac{\partial\Pi^{12}}{\partial x^1},\quad T_2=\Pi^{12}\frac{\partial\eta_2}{\partial x^2}-\eta_1\frac{\partial\Pi^{12}}{\partial x^1},\\
T_3&=\Pi^{12}\frac{\partial\eta_1}{\partial x^1}-\eta_2\frac{\partial\Pi^{12}}{\partial x^2},\quad T_4=\Pi^{12}\frac{\partial\eta_2}{\partial x^1}+\eta_1\frac{\partial\Pi^{12}}{\partial x^2},\\
T_5&=\eta_1\Pi^{12}\frac{\partial\eta_1}{\partial x^2}-\eta_2\Pi^{12}\frac{\partial\eta_1}{\partial x^1}+\eta_1\eta_2\frac{\partial\Pi^{12}}{\partial x^1}+\eta_2^2\frac{\partial\Pi^{12}}{\partial x^2},\\
T_6&=\eta_2\Pi^{12}\frac{\partial\eta_2}{\partial x^1}-\eta_1\Pi^{12}\frac{\partial\eta_2}{\partial x^2}+\eta_1\eta_2\frac{\partial\Pi^{12}}{\partial x^2}+\eta_1^2\frac{\partial\Pi^{12}}{\partial x^1}.
\end{align*}
\end{theorem}
\begin{proof}
Since $\{dx^1,dx^2\}$ is orthonormal coframe field on $\mathbb{R}^2$ therefore $(4.6)$ implies that 
\begin{equation}
2\mathcal{R}(\eta,dx^1,dx^2,\eta)=2g(\mathcal{D}_{dx^1}\eta,\mathcal{D}_{dx^2}\eta)+g(\mathcal{D}_{dx^1}\mathcal{D}_\eta\eta,dx^2)+g(\mathcal{D}_{dx^2}\mathcal{D}_\eta\eta,{dx^1}).
\end{equation} 
The local components of $\tilde{g}$ are given by
\begin{eqnarray}
\left\{
\begin{array}{ll}
\tilde{g}_{11}=\tilde{g}(\frac{\partial}{\partial x^1},\frac{\partial}{\partial x^1})=1,\\
\tilde{g}_{22}=\tilde{g}(\frac{\partial}{\partial x^2},\frac{\partial}{\partial x^2})=1,\\
\tilde{g}_{12}=\tilde{g}(\frac{\partial}{\partial x^1},\frac{\partial}{\partial x^2})=0.
\end{array}
\right.
\end{eqnarray}
As $g$ is the cometric of the metric $\tilde{g}$ then its local components are given by
\begin{eqnarray}
\left\{
\begin{array}{ll}
g^{11}=g(dx^1,dx^1)=1,\\
g^{22}=g(dx^2,dx^2)=1,\\
g^{12}=g(dx^1,dx^2)=0.
\end{array}
\right.
\end{eqnarray}
Now, from (\cite{bp},eqn. 6.2), Christoffel symbols $\Gamma_k^{ij}$ (where $i,j,k\in\{1,2\}$) defined as
\begin{align}
\Gamma_k^{ij}&=\frac{1}{2}\sum_{l}\sum_{m}g_{mk}\Big(\Pi^{il}\frac{\partial g^{jm}}{\partial x_l}+\Pi^{jl}\frac{\partial g^{im}}{\partial x_l}-\Pi^{ml}\frac{\partial g^{ij}}{\partial x_l}-g^{li}\frac{\partial\Pi^{jm}}{\partial x_l}-g^{lj}\frac{\partial\Pi^{im}}{\partial x_l}\Big)\nonumber\\
&+\frac{1}{2}\frac{\partial\Pi^{ij}}{\partial x_k}.
\end{align}
Therefore from $(4.10)$ and $(4.11)$, we have 
\begin{align}
&\Gamma_1^{11}=0,\quad\Gamma_1^{12}=\frac{\partial\Pi^{12}}{\partial x^1},\quad\Gamma_1^{21}=0,\quad\Gamma_1^{22}=\frac{\partial\Pi^{12}}{\partial x^2},\nonumber\\
&\Gamma_2^{11}=-\frac{\partial\Pi^{12}}{\partial x^1},\quad\Gamma_2^{12}=0,\quad\Gamma_2^{21}=-\frac{\partial\Pi^{12}}{\partial x^2},\quad\Gamma_2^{22}=0.
\end{align}
Hence $(4.7)$ and $(4.12)$, conclude that
\begin{align}
&\mathcal{D}_{dx^1}dx^1=-\frac{\partial\Pi^{12}}{\partial x^1}dx^2,\quad\mathcal{D}_{dx^1}dx^2=\frac{\partial\Pi^{12}}{\partial x^1}dx^1,\nonumber\\
&\mathcal{D}_{dx^2}dx^1=-\frac{\partial\Pi^{12}}{\partial x^2}dx^2,\quad\mathcal{D}_{dx^2}dx^2=\frac{\partial\Pi^{12}}{\partial x^2}dx^1.
\end{align}  
By the properties of contravariant Levi-Civita connection $\mathcal{D}$ and equation $(4.13)$, we have
\begin{align}
\mathcal{D}_{dx^1}\eta&=T_1dx^1+T_2dx^2,\\
\mathcal{D}_{dx^2}\eta&=-T_3dx^1-T_4dx^2,\\
\mathcal{D}_{\eta}\eta&=T_5dx^1-T_6dx^2.
\end{align}
Equations $(4.14)$, $(4.15)$ and $(4.10)$, provides
\begin{equation}
g(\mathcal{D}_{dx^1}\eta,\mathcal{D}_{dx^2}\eta)=-T_1T_3-T_2T_4.
\end{equation}
Equations $(4.16)$ and $(4.10)$, provides
\begin{align}
g(\mathcal{D}_{dx^1}\mathcal{D}_{\eta}\eta,dx^2)&=-T_5\frac{\partial\Pi^{12}}{\partial x^1}-\Pi^{12}\frac{\partial T_6}{\partial x^2},\\
g(\mathcal{D}_{dx^1}\mathcal{D}_{\eta}\eta,dx^2)&=-T_6\frac{\partial\Pi^{12}}{\partial x^2}-\Pi^{12}\frac{\partial T_5}{\partial x^1}.
\end{align}
Using equations $(4.17)$, $(4.18)$ and $(4.19)$ in $(4.8)$, proves this result.
\end{proof}
\textbf{Appendix. Proof of Proposition 4.5}\\
Equation $(4.2)$, is given by
\begin{align}
(\mathcal{L}_\eta\mathcal{L}_\eta g)(\alpha,\beta)&=g(\mathcal{D}_\eta\mathcal{D}_\alpha\eta,\beta)+g(\mathcal{D}_\eta\mathcal{D}_\beta\eta,\alpha)\nonumber\\
&-g(\mathcal{D}_{[\eta,\alpha]_\Pi}\eta,\beta)-g(\mathcal{D}_{[\eta,\beta]_\Pi}\eta,\alpha)+2g(\mathcal{D}_\alpha\eta,\mathcal{D}_\beta\eta).
\end{align}
Using $(2.5)$ and Proposition 2.1 in the first term $P_1$ of $(4.20)$, we have
\begin{align*}
P_1&=g(\mathcal{D}_\eta\mathcal{D}_\alpha\eta,\beta)\\
&=g(\mathcal{D}_\eta\mathcal{D}_\alpha\eta,\beta_1^h)+g(\mathcal{D}_\eta\mathcal{D}_\alpha\eta,\beta_2^v).
\end{align*}
Assume that $S_1=\mathcal{D}_\eta\mathcal{D}_\alpha\eta$, therefore  
\begin{align*}
S_1&=\mathcal{D}_{{\eta_1}^h}\mathcal{D}_\alpha\eta+\mathcal{D}_{{\eta_2}^v}\mathcal{D}_\alpha\eta\\
&=(\mathcal{D}_{\eta_1}^1\mathcal{D}_{\alpha_1}^1\eta_1)^h+(\mathcal{D}_{\eta_2}^2\mathcal{D}_{\alpha_2}^2\eta_2)^v-\big(\frac{\mathcal{D}_{\eta_1}^1J_1df}{f^3}\big)^hg_2(\alpha_2,\eta_2)^v+\big(\frac{g_1(J_1df,\alpha_1)}{f}\big)^h(\mathcal{D}_{\eta_2}^2\eta_2)^v\\
&+\big(\frac{g_1(J_1df,\eta_1)}{f}\big)^h(\mathcal{D}_{\alpha_2}^2\eta_2+\mathcal{D}_{\eta_2}^2\alpha_2)^v+\Big[\big(\frac{3(\mathcal{D}_{\eta_1}^1f)-g_1(J_1df,\eta_1)}{f^4}\big)^hg_2(\alpha_2,\eta_2)^v\\
&+\big(\frac{g_1(J_1df,\alpha_1)}{f^4}\big)^h(||\eta_2||_2^2)^v-\frac{1}{(f^h)^3}g_2(\mathcal{D}_{\alpha_2}^2\eta_2,\eta_2)^v-\frac{1}{(f^h)^3}\big(\mathcal{D}_{\eta_2}^2g_2(\alpha_2,\eta_2)\big)^v\Big](J_1df)^h\\
&+\Big[\frac{g_1(J_1df,\eta_1)^2}{f^2}-\frac{(\mathcal{D}_{\eta_1}^1f)g_1(J_1df,\eta_1)}{f^2}+\frac{\mathcal{D}_{\eta_1}^1g_1(J_1df,\eta_1)}{f}\Big]^h\alpha_2^v+\Big[\big(\frac{\mathcal{D}_{\eta_1}^1g_1(J_1df,\alpha_1)}{f}\big)^h\\
&+\big(\frac{g_1(J_1df,\mathcal{D}_{\alpha_1}^1\eta_1)}{f}\big)^h+\big(\frac{g_1(J_1df,\alpha_1)g_1(J_1df,\eta_1)}{f^2}\big)^h-\big(\frac{(\mathcal{D}_{\eta_1}^1f)g_1(J_1df,\alpha_1)}{f^2}\big)^h\\
&-\big(\frac{||J_1df||_1^2}{f^4}\big)^hg_2(\alpha_2,\eta_2)^v\Big]\eta_2^v.
\end{align*}
Using $S_1$ in $P_1$, provides
\begin{align*}
P_1&=g_1(\mathcal{D}_{\eta_1}^1\mathcal{D}_{\alpha_1}^1\eta_1,\beta_1)^h+\frac{1}{(f^h)^2}g_2(\mathcal{D}_{\eta_2}^2\mathcal{D}_{\alpha_2}^2\eta_2,\beta_2)^v-\big(\frac{g_1(\mathcal{D}_{\eta_1}^1J_1df,\beta_1)}{f^3}\big)^hg_2(\alpha_2,\eta_2)^v\\
&+\big(\frac{g_1(J_1df,\eta_1)}{f^3}\big)^hg_2(\mathcal{D}_{\alpha_2}^2\eta_2+\mathcal{D}_{\eta_2}^2\alpha_2,\beta_2)^v+\big(\frac{g_1(J_1df,\alpha_1)}{f^3}\big)^hg_2(\mathcal{D}_{\eta_2}^2\eta_2,\beta_2)^v\\
&+\Big[\big(\frac{g_1(J_1df,\alpha_1)}{f^4}\big)^h(||\eta_2||_2^2)^v-\big(\frac{g_1(J_1df,\eta_1)}{f^4}\big)^hg_2(\eta_2,\alpha_2)^v+3\big(\frac{\mathcal{D}_{\eta_1}^1f}{f^4}\big)^hg_2(\alpha_2,\eta_2)^v\\
&-\frac{1}{(f^h)^3}g_2(\mathcal{D}_{\alpha_2}^2\eta_2,\eta_2)^v-\frac{1}{(f^h)^3}\big(\mathcal{D}_{\eta_2}^2g_2(\alpha_2,\eta_2)\big)^v\Big]g_1(J_1df,\beta_1)^h+\Big[\frac{g_1(J_1df,\eta_1)^2}{f^4}\\
&-\frac{(\mathcal{D}_{\eta_1}^1f)g_1(J_1df,\eta_1)}{f^4}+\frac{\mathcal{D}_{\eta_1}^1g_1(J_1df,\eta_1)}{f^3}\Big]^hg_2(\alpha_2,\beta_2)^v+\Big[\big(\frac{\mathcal{D}_{\eta_1}^1g_1(J_1df,\alpha_1)}{f^3}\big)^h\\
&+\big(\frac{g_1(J_1df,\mathcal{D}_{\alpha_1}^1\eta_1)}{f^3}\big)^h+\big(\frac{g_1(J_1df,\alpha_1)g_1(J_1df,\eta_1)}{f^4}\big)^h-\big(\frac{(\mathcal{D}_{\eta_1}^1f)g_1(J_1df,\alpha_1)}{f^4}\big)^h\\
&-\big(\frac{||J_1df||_1^2}{f^6}\big)^hg_2(\alpha_2,\eta_2)^v\Big]g_2(\eta_2,\beta_2)^v.
\end{align*}
After exchanging $\alpha$ and $\beta$ in the last equation provides the second term $P_2$ of $(4.20)$, is given by
\begin{align*}
P_2&=g(\mathcal{D}_\eta\mathcal{D}_\alpha\eta,\beta)\\
&=g_1(\mathcal{D}_{\eta_1}^1\mathcal{D}_{\beta_1}^1\eta_1,\alpha_1)^h+\frac{1}{(f^h)^2}g_2(\mathcal{D}_{\eta_2}^2\mathcal{D}_{\beta_2}^2\eta_2,\alpha_2)^v-\big(\frac{g_1(\mathcal{D}_{\eta_1}^1J_1df,\alpha_1)}{f^3}\big)^hg_2(\beta_2,\eta_2)^v\\
&+\big(\frac{g_1(J_1df,\eta_1)}{f^3}\big)^hg_2(\mathcal{D}_{\beta_2}^2\eta_2+\mathcal{D}_{\eta_2}^2\alpha_2,\alpha_2)^v+\big(\frac{g_1(J_1df,\beta_1)}{f^3}\big)^hg_2(\mathcal{D}_{\eta_2}^2\eta_2,\alpha_2)^v\\
&+\Big[\big(\frac{g_1(J_1df,\beta_1)}{f^4}\big)^h(||\eta_2||_2^2)^v-\big(\frac{g_1(J_1df,\eta_1)}{f^4}\big)^hg_2(\eta_2,\beta_2)^v+3\big(\frac{\mathcal{D}_{\eta_1}^1f}{f^4}\big)^hg_2(\beta_2,\eta_2)^v\\
&-\frac{1}{(f^h)^3}g_2(\mathcal{D}_{\beta_2}^2\eta_2,\eta_2)^v-\frac{1}{(f^h)^3}\big(\mathcal{D}_{\eta_2}^2g_2(\beta_2,\eta_2)\big)^v\Big]g_1(J_1df,\alpha_1)^h+\Big[\frac{g_1(J_1df,\eta_1)^2}{f^4}\\
&-\frac{(\mathcal{D}_{\eta_1}^1f)g_1(J_1df,\eta_1)}{f^4}+\frac{\mathcal{D}_{\eta_1}^1g_1(J_1df,\eta_1)}{f^3}\Big]^hg_2(\beta_2,\alpha_2)^v+\Big[\big(\frac{\mathcal{D}_{\eta_1}^1g_1(J_1df,\beta_1)}{f^3}\big)^h\\
&+\big(\frac{g_1(J_1df,\mathcal{D}_{\beta_1}^1\eta_1)}{f^3}\big)^h+\big(\frac{g_1(J_1df,\beta_1)g_1(J_1df,\eta_1)}{f^4}\big)^h-\big(\frac{(\mathcal{D}_{\eta_1}^1f)g_1(J_1df,\beta_1)}{f^4}\big)^h\\
&\big(\frac{||J_1df||_1^2}{f^6}\big)^hg_2(\beta_2,\eta_2)^v\Big]g_2(\eta_2,\alpha_2)^v.
\end{align*}
Again using $(2.5)$ and Proposition 2.1 in the third term $P_3$ of $(4.20)$, we have
\begin{align*}
P_3&=g(\mathcal{D}_{[\eta,\alpha]_\Pi}\eta,\beta)\\
&=g(\mathcal{D}_{[\eta,\alpha]_\Pi}\eta,\beta_1^h)+g(\mathcal{D}_{[\eta,\alpha]_\Pi}\eta,\beta_2^v).\\
\end{align*}
Assume that $S_2=\mathcal{D}_{[\eta,\alpha]_\Pi}\eta$, therefore
\begin{align*}
S_2&=\mathcal{D}_{[\eta,\alpha]_\Pi}\eta_1^h+\mathcal{D}_{[\eta,\alpha]_\Pi}\eta_2^v\\
&=\mathcal{D}_{[\eta_1,\alpha_1]_{\Pi_1}^h}\eta_1^h+\mathcal{D}_{[\eta_1,\alpha_1]_{\Pi_1}^h}\eta_2^v+\mathcal{D}_{[\eta_2,\alpha_2]_{\Pi_2}^v}\eta_1^h+\mathcal{D}_{[\eta_2,\alpha_2]_{\Pi_2}^v}\eta_2^v\\
&=(\mathcal{D}_{[\eta_1,\alpha_1]_{\Pi_1}}^1\eta_1)^h+(\mathcal{D}_{[\eta_2,\alpha_2]_{\Pi_2}}^2\eta_2)^v+\big(\frac{g_1(J_1df,\eta_1)}{f}\big)^h[\eta_2,\alpha_2]_{\Pi_2}^v\\
&+\big(\frac{g_1(J_1df,[\eta_1,\alpha_1]_{\Pi_1})}{f}\big)^h\eta_2^v-\big(\frac{J_1df}{f^3}\big)^hg_2(\eta_2,[\eta_2,\alpha_2]_{\Pi_2})^v.
\end{align*}
Using $S_2$ in $P_3$, provides
\begin{align*}
P_3&=g_1(\mathcal{D}_{[\eta_1,\alpha_1]_{\Pi_1}}^1\eta_1,\beta_1)^h+\frac{1}{(f^h)^2}g_2(\mathcal{D}_{[\eta_2,\alpha_2]_{\Pi_2}}^2\eta_2,\beta_2)^v\\
&+\big(\frac{g_1(J_1df,\eta_1)}{f^3}\big)^hg_2([\eta_2,\alpha_2]_{\Pi_2},\beta_2)^v+\big(\frac{g_1(J_1df,[\eta_1,\alpha_1]_{\Pi_1})}{f^3}\big)^hg_2(\eta_2,\beta_2)^v\\
&-\big(\frac{g_1(J_1df,\beta_1)}{f^3}\big)^hg_2(\eta_2,[\eta_2,\alpha_2]_{\Pi_2})^v.
\end{align*}
After exchanging $\alpha$ and $\beta$ in the last equation provides the fourth term $P_4$ of $(4.20)$, is given by
\begin{align*}
P_4&=g_1(\mathcal{D}_{[\eta_1,\beta_1]_{\Pi_1}}^1\eta_1,\alpha_1)^h+\frac{1}{(f^h)^2}g_2(\mathcal{D}_{[\eta_2,\beta_2]_{\Pi_2}}^2\eta_2,\alpha_2)^v\\
&+\big(\frac{g_1(J_1df,\eta_1)}{f^3}\big)^hg_2([\eta_2,\beta_2]_{\Pi_2},\alpha_2)^v+\big(\frac{g_1(J_1df,[\eta_1,\beta_1]_{\Pi_1})}{f^3}\big)^hg_2(\eta_2,\alpha_2)^v\\
&-\big(\frac{g_1(J_1df,\alpha_1)}{f^3}\big)^hg_2(\eta_2,[\eta_2,\beta_2]_{\Pi_2})^v.
\end{align*}
Same as above manipulations the fifth term $P_5$ of $(4.20)$, is given by
\begin{align*}
P_5&=g_1(\mathcal{D}_{\alpha_1}^1\eta_1,\mathcal{D}_{\beta_1}^1\eta_1)^h+\frac{1}{(f^h)^2}g_2(\mathcal{D}_{\alpha_2}^2\eta_2,\mathcal{D}_{\beta_2}^2\eta_2)^v+\big(\frac{g_1(J_1df,\alpha_1)}{f^3}\big)^hg_2(\mathcal{D}_{\beta_2}^2\eta_2,\eta_2)^v\\
&+\big(\frac{g_1(J_1df,\beta_1)}{f^3}\big)^hg_2(\mathcal{D}_{\alpha_2}^2\eta_2,\eta_2)^v+\big(\frac{||J_1df||_1^2}{f^6}\big)^hg_2(\alpha_2,\eta_2)^vg_2(\beta_2,\eta_2)^v\\
&+\big(\frac{g_1(J_1df,\alpha_1)g_1(J_1df,\beta_1)}{f^4}\big)^h(||\eta_2||_2^2)^v+\big(\frac{g_1(J_1df,\eta_1)^2}{f^4}\big)^hg_2(\alpha_2,\beta_2)^v\\
&+\Big[\frac{g_1(J_1df,\alpha_1)g_1(J_1df,\eta_1)}{f^4}-\frac{g_1(J_1df,\mathcal{D}_{\alpha_1}^1\eta_1)}{f^3}\Big]^hg_2(\beta_2,\eta_2)^v\\
&+\Big[\frac{g_1(J_1df,\beta_1)g_1(J_1df,\eta_1)}{f^4}-\frac{g_1(J_1df,\mathcal{D}_{\beta_1}^1\eta_1)}{f^3}\Big]^hg_2(\alpha_2,\eta_2)^v\\
&+\big(\frac{g_1(J_1df,\eta_1)}{f^3}\big)^h\big(g_2(\mathcal{D}_{\alpha_2}^2\eta_2,\beta_2)+g_2(\mathcal{D}_{\beta_2}^2\eta_2,\alpha_2)\big)^v.
\end{align*}
Using terms $P_1$, $P_2$, $P_3$, $P_4$ and $P_5$ in $(4.20)$ and after some manipulations provide the result.

\end{document}